\documentclass[letterpaper, 10 pt, conference]{ieeeconf}
\IEEEoverridecommandlockouts
\newcommand{\mat}[1]{\begin{matrix}#1\end{matrix}}

\usepackage{amsmath,amssymb}
\usepackage{cite}
\usepackage{amsmath,amssymb,amsfonts}
\usepackage{algorithmic}
\usepackage{graphicx}
\usepackage{textcomp}
\usepackage{xcolor}
\usepackage{comment}
\usepackage{subcaption}

\makeatletter
\def\@begintheorem#1#2{\trivlist
\item[\hskip\labelsep{\bfseries #1\ #2.}]\itshape}
\def\@opargbegintheorem#1#2#3{\trivlist
\item[\hskip\labelsep{\bfseries #1\ #2\ (#3).}]\itshape}
\makeatother
\newtheorem{theorem}{Theorem}
\newtheorem{proposition}[theorem]{Proposition}
\newtheorem{corollary}{Corollary}[theorem]
\newtheorem{lemma}[theorem]{Lemma}
\newtheorem{definition}{Definition}

\newtheorem{remark}{Remark}

\newcommand{\bmtx}{\begin{bmatrix}}
\newcommand{\emtx}{\end{bmatrix}}
\newcommand{\bsmtx}{\left[ \begin{smallmatrix}} 
\newcommand{\esmtx}{\end{smallmatrix} \right]}

\newcommand{\R}{\field{R}}

\renewcommand{\R}{\mathbb{R}}
\newcommand{\mcl}[1]{\mathcal{#1}}

\def\BibTeX{{\rm B\kern-.05em{\sc i\kern-.025em b}\kern-.08em
    T\kern-.1667em\lower.7ex\hbox{E}\kern-.125emX}}
\title{\LARGE \bf Parameterization of Seed Functions for Equivalent Representations of 
            Time-Varying Delay Systems}
\author{Sengiyumva Kisole$^{1}$, Jungbae Chun$^{2}$, Peter Seiler$^{2}$, and Matthew M.~Peet$^{1}$%
\thanks{This material is based upon work supported by the National Science Foundation under Grants NSF EPCN-2337751 and NSF EPCN-2337752.}
\thanks{$^{1}$School for the Engineering of Matter, Transport and Energy, Arizona State University, Tempe, AZ 85298, USA.
{\tt\small sengi.kisole@asu.edu, mpeet@asu.edu}}%
\thanks{$^{2}$Department of Electrical Engineering and Computer Science, University of Michigan, Ann Arbor, MI 48109, USA.
{\tt\small jungbaec@umich.edu, pseiler@umich.edu}}%
}
\begin{document}
\maketitle
\thispagestyle{empty}
\pagestyle{empty}

\begin{abstract}
Abel's classic transformation shows that any well-posed system with time-varying delay is equivalent to a parameter-varying system with fixed delay. The existence of such a parameter-varying constant delay representation then simplifies the problems of stability analysis and optimal control. Unfortunately, the method for construction of such transformations has been ad-hoc -- requiring an iterative time-stepping approach to constructing the transformation beginning with a seed function subject to boundary-value constraints. Moreover, a poor choice of seed function often results in a constant delay representati on with large time-variations in system parameters -- obviating the benefits of such a representation. In this paper, we show how the set of all feasible seed functions can be parameterized using a basis for $L_2$. This parameterization is then used to search for seed functions for which the corresponding time-transformation results in smaller parameter variation. 
The parameterization of admissible seed functions is illustrated with numerical examples that contrast how well-chosen and poorly chosen seed functions affect the boundedness of a time transformation.
\end{abstract}
\bstctlcite{IEEEexample:BSTcontrol}
\section{Introduction}
Time-varying delays appear in networked control, teleoperation, and biological regulation, where even small time-variations can alter stability and performance. However, analysis and control of systems with time-varying delay is complicated by the typically associated time-varying notion of state-space~\cite{michiels2011systems, verriest2012SS3, verriest2010well,verriest2011inconsistencies}.
To address this issue, a classical remedy, dating back to Abel-type transformations, shows that any well-posed system with a time-varying delay can be recast as a parameter-varying system with a fixed delay; with a change of time variables, one regains a time-invariant state space~\cite{heard1975change,vcermak1995continuous,neuman1990transformation,neuman1981transformations,neuman1982simultaneous,an2000note,brunner2009time1}. In the fixed-delay representation, standard tools— Linear Matrix Inequalities (LMIs), Integral Quadratic Constraints (IQCs), and  the Partial Integral Equations (PIEs) framework, can then be used for analysis and control design~\cite{pfifer2015overview, shivakumar2021pietools2, talitckii2023integral, megretski2002system}.


The conversion of a system with time varying delay to parameter-varying system with fixed delay is based on a given time-transformation $t \mapsto h(t)$ which must satisfy certain recursive constraints (Lem.~\ref{Cor:Equiv}). However, this time-transformation is not unique, but rather is parameterized by a choice of \textit{seed function}, $\phi$ which specifies the transformation on an initial time interval. This seed function must itself satisfy certain boundary constraints as determined by the time-varying delay. Then, for any valid seed function, the corresponding time-transformation may be constructed iteratively. 

Unfortunately, however, as is demonstrated in Section~\ref{sec6}, not all seed functions result in well-behaved time-transformations. Specifically, we show that many choices of seed function result in a time-transformation whose derivative, $\dot h$ grows over time. This growth in $\dot h$ translates directly into large parameter variations in the resulting fixed-delay representation -- an attempt to quantify this growth for systems with periodic time-varying delays via perturbative expansion can be found in~\cite{Chun2025Pertubative}. This parameter variation then prevents the use of techniques such as normalization~\cite{brunner2009time1,nah2020normalization} to the fixed-delay representation. 

The goal of this paper, then, is to propose a parameterization of seed functions which then allows one to find time-transformations with desirable properties such as periodicity and boundedness. Specifically, we provide an $L_2$ parameterization of the set of admissible seed functions ($S_\tau$), defining a affine operator $\mcl T:L_2 \rightarrow S$ such that for any $\nu \in L_2$, $\phi =\mcl T \nu \in S$. Furthermore, we show that for any $\phi \in S$, $\mcl T \phi'''=\phi$ -- establishing a one-to-one map from $L_2$ to $S$. By working directly with the seed parameters, $\nu$, rather than the seed itself, one can enforce constraints on the seed functions (and resulting time-transformations). For example, using a polynomial basis for $L_2$, one can use Sum of Squares (SoS)~\cite{parrilo2000structured} to enforce non-negativity of $\phi'$.

To establish the map from $L_2$ to seed functions, $\phi$, to time-transformation, $h$, to constant delay representation, we begin in Section~\ref{sec:equivalence} by defining the class of admissible time-transformations  and showing that for any such time-transformation, solutions of the resulting fixed-delay and time-varying delay systems are equivalent -- a slight extension of classical results in~\cite{heard1975change, vcermak1995continuous, neuman1990transformation}. We then define the set of admissible seed functions and show how any such seed function defines a resulting time-transformation, $h$. 

%
%

In Section~\ref{sec5}, we then establish the invertible mapping between seed parameters and seed functions -- also showing that the well-known quadratic seed function is a special case and that monotonicity of the seed function can be enforced using SoS constraints on the seed parameter.

Finally, in Section~\ref{sec6}, we motivate the proposed parameterization by demonstrating the impact of a seed function choice on the derivative of the time transformation, $\dot h$ --- comparing the effect of quadratic, exponential, and affine plus sinusoidal seed functions on $\dot h$ for a time-varying sinusoidal delay.

\section{Notation}
$\mathbb{N}^n$, $\R^n$, and $\R^n_+$ denote the space of $n-$dimensional vectors of  natural, real, and positive real numbers, respectively. 
$L_2([a,b],\R)$ is the space of square–integrable functions; $\mcl H^3([a,b],\R)$ is the Sobolev space of functions with third (weak) derivative in $L_2$. 
Composition is written $(g\circ f)(x)=g(f(x))$. Moreover, $f^{\circ k}$ denotes $k$-fold composition, and $(f^{-1})^{\circ k}$ the $k$-fold composition of the inverse (when it exists). For compact $\Omega\subset\R$, $\mcl C^k(\Omega)$ denotes the space of $k$-times continuously differentiable functions $f:\Omega\to\R$ with norm 
$\|f\|=\sup_{t\in\Omega}\|f(t)\|_2$.

\section{Equivalence Between Variable- and Fixed-Delay DDEs}\label{sec:equivalence}
Consider a delay differential equation (DDE) of the form:
\vspace{-2mm}
\begin{align}
\label{eqn:dde_var}
\dot{x}(t) &= A_0 x(t) + A_1 x\!\big(t-\tau(t)\big), \quad t\ge 0,\\
x(t) &= \zeta(t), \quad t\in[-\tau(0),\,0], \nonumber
\end{align}
where $A_0$, $A_1 \in \mathbb{R}^{n_x \times n_x}$. We assume the initial function \(\zeta \in \mathcal{C}([-\tau(0), 0], \mathbb{R}^{n_x})\) is continuous, and the time-varying delay \(\tau \in \mathcal{C}^1(\mathbb{R}, \mathbb{R}_+)\) is bounded with \(\dot{\tau}(t) < 1\) for all \(t \geq 0\).

\subsection{Time Transformations} \label{Equivalence}
A \emph{time transformation}, $h$, may be used to convert a DDE with time varying delay to a parameter-varying DDE with fixed time delay. This function is invertible and changes the time variable $t$ to a new time variable defined as $\lambda:=h^{-1}(t)$. 
The time transformation $h(\lambda)$, must be a strictly increasing function constructed to satisfy an Abel equation. This Abel equation ensures that the delay $ \tau(t) $ at time $ t = h(\lambda) $ aligns with a constant shift $ \tau^* $ in the new time $ \lambda $. This is formalized in the following lemma which is a slight modification of \cite[Thm.~1]{nah2020normalization}.

\begin{lemma}
\label{Cor:Equiv}
Suppose $\tau^*\in \mathbb{R}_+$ and $\tau \in \mathcal{C}^1(\mathbb{R}, \mathbb{R}_+)$  with $\dot{\tau}(t) < 1$. Let $h \in \mathcal{C}^1([-\tau^*,\infty), [-\tau(0), \infty))$ be a strictly increasing unbounded function,  with $h(0)=0$ and where 
\begin{equation}
\label{eqn:abel0}
 h(\lambda) - \tau(h(\lambda)) = h(\lambda - \tau^*), \quad \forall \lambda \geq 0.
\end{equation}
For any $\zeta \in \mathcal{C}[-\tau(0),0],$ if $x(t)$ satisfies
	\begin{equation}
		\label{eqn:01}
		\begin{aligned}
			\dot{x}(t) &= A_0 x(t) + A_1 x(t - \tau(t)), \quad \forall t \geq 0 \\
			x(t)&= \zeta (t), \quad \forall t \in [\,-\tau(0),\,0\,]
		\end{aligned}
	\end{equation}
then $\bar{x}(\lambda)= x(h(\lambda))$  and $\bar \zeta = \zeta \big(h(\lambda)\big)$ satisfy
	\begin{equation}
		\label{eqn:02}
		\begin{aligned}
			\dot{\bar{x}}(\lambda) &=  \dot{h}(\lambda) A_0 \bar{x}(\lambda) + \dot{h}(\lambda) A_1 \bar{x}(\lambda -\tau^*), \quad \forall \lambda \geq 0 \\
			\bar{x}(\lambda) &= \bar \zeta\!\big(\lambda\big), \quad \forall \lambda \in [-\tau^*,0].
		\end{aligned}
	\end{equation}
Conversely, if $\bar \zeta \in \mathcal{C}[-\tau^*,0],$ and $\bar x$ satisfies Eqn.~\eqref{eqn:02}, then $x(t) =\bar{x}(h^{-1}(t))$  and $\zeta \big(t\big)=\bar \zeta (h^{-1}(t)) $ satisfy Eqn.~\eqref{eqn:01}.
\end{lemma}

\begin{proof} 
By the inverse function theorem, $h^{-1} \in \mcl C^{1}$ and
\[
(h^{-1})'(t)=\frac{1}{\dot h\!\big(h^{-1}(t)\big)},\qquad t\in[-\tau(0),\infty).
\]
\emph{($\Rightarrow$)} Let $\zeta\in\mathcal{C}([-\tau(0),0])$ and suppose $x$ solves \eqref{eqn:01}. 
Define $\bar{x}(\lambda):=x\!\big(h(\lambda)\big)$. For $\lambda\ge 0$,
\begin{align*}
\dot{\bar{x}}(\lambda) &=\dot{x}\!\big(h(\lambda)\big)\dot h(\lambda) \\
&=\dot h(\lambda)\Big[A_0 x\!\big(h(\lambda)\big)+A_1 x\!\big(h(\lambda)-\tau(h(\lambda))\big)\Big].
\end{align*}
Using~\eqref{eqn:abel0}, we have:
\(
\dot{\bar{x}}(\lambda)=\dot h(\lambda)\Big[A_0 \bar{x}(\lambda)+A_1 \bar{x}(\lambda-\tau^*)\Big].
\) For the initial segment, if $\lambda\in[-\tau^*,0]$ then 
$h(\lambda)\in[h(-\tau^*),h(0)]=[-\tau(0),0]$, hence
\(
\bar{x}(\lambda)=x\!\big(h(\lambda)\big)=\zeta\!\big(h(\lambda)\big).
\)
Thus \eqref{eqn:02} holds.

\emph{($\Leftarrow$)} Let $\bar \zeta\in\mathcal{C}([-\tau^*,0])$ and suppose $\bar {x}$ satisfies \eqref{eqn:02}.
Define $x(t):=\bar {x}\!\big(h^{-1}(t)\big)$ for $t\ge 0$. Then
\begin{align*}
\dot{x}(t)&=\dot{\bar {x}}(h^{-1}(t))\,(h^{-1})'(t) 
=\dot{\bar{x}}(h^{-1}(t))\,\frac{1}{\dot h(h^{-1}(t))}\\
&=A_0 \bar{x}(h^{-1}(t))+A_1 \bar{x}(h^{-1}(t)-\tau^*)
\end{align*}
From~\eqref{eqn:abel0} , we obtain
\[
h^{-1}\!\big(t-\tau(t)\big)=h^{-1}(t)-\tau^*,
\]
so $\bar{x}(h^{-1}(t)-\tau^*)=\bar{x}\!\big(h^{-1}(t-\tau(t))\big)=x\!\big(t-\tau(t)\big)$. Hence
\(
\dot{x}(t)=A_0 x(t)+A_1 x\!\big(t-\tau(t)\big),
\)
which is the differential equation in \eqref{eqn:01}. For the initial segment, if $t\in[-\tau(0),0]$ then 
for $h^{-1}(t)\in[-\tau^*,0]$, and thus for  $\zeta(t)=\bar \zeta(h^{-1}(t))$,
\[
x(t)=\bar{x}\!\big(h^{-1}(t)\big)
=\bar \zeta\!\big(h^{-1}(t)\big)=:\zeta(t),
\]
so \eqref{eqn:01} holds.
\end{proof}

Lemma 1 establishes equivalence of solutions between a system with time-varying delay and a fixed-delay system with parameter-varying uncertainty. We note, however, that the choice of time-transformation, $h$, may not be uniquely defined. Nor have we suggested any approach to constructing such a transformation. In the following subsection, we show that such time-transformations may be parameterized through the choice of seed function, $\phi$.

\subsection{Recursive Construction of the Time Transformation}\label{sec3A}
Lemma~\ref{Cor:Equiv} defines a map from a set of admissible time-transformations to a set of equivalent fixed-delay representations of a system with time-varying delay. In this subsection, we show how such time-transformations may be parameterized by choice of admissible seed function, $\phi$.

The set of admissible seed functions is defined next based on ~\cite[Thm. 1]{nah2020normalization}, \cite[Sec. 2.3]{brunner2009time1}.
\begin{definition}
\label{def:seed}
We say that $\phi$ is a admissible seed function associated to $\tau^*, \tau(0) >0,  \dot \tau(0)< 1$ if $\phi$ is strictly increasing and $\phi \in S_{\dot \tau(0),\tau(0),\tau^*}$ where 
\begin{align}
\label{seed_set}
& S_{\dot \tau_0,\tau_0,\tau^*}:=\\
&\left\{\phi \in \mcl H^3[-\tau^*,0]\,:\, \mat{\phi(0)=0,\\
\phi(-\tau^*) =-\tau_0},\; \dot \phi(0)=\frac{\dot \phi(-\tau^*)}{\,1-\dot \tau_0} \right\} \notag
\end{align}


\end{definition}

For any given admissible seed function, we may construct an associated time transformation $h$ using the recursion defined in \cite{nah2020normalization, brunner2009time1}. Specifically, given $\tau(t)$ and  $\phi \in S_{\dot \tau(0),\tau(0),\tau^*}$, let $h \in \mcl{C}^1([-\tau^*, \infty), \R_+)$  be defined as:
\begin{equation}
\label{recur_con}
h(\lambda) :=
\begin{cases}
\phi(\lambda), & \lambda \in [-\tau^*, 0] \\
\theta^{-1}\big(h(\lambda - \tau^*)\big), & \lambda \geq 0
\end{cases}
\end{equation}
where $\theta(t):=t-\tau(t)$ is strictly increasing and unbounded and hence is invertible.
Equivalently, for $k \in \mathbb{N}$, 
\begin{equation}
\label{eqn:Recur}
h(\lambda) := (\theta^{-1})^{\circ k} \bigl( \phi(\lambda - k\tau^*) \bigr) \qquad \lambda \in [(k-1)\tau^*, k\tau^*].
\end{equation}
Moreover, for $k \in \mathbb{N}$ and $\lambda \in [(k-1)\tau^*, k\tau^*]$,
\begin{equation}
\label{eq:deriv-prop1}
\dot{h}(\lambda) 
= \frac{\dot{\phi}(\lambda - k\tau^*)}
     {\prod_{n=1}^{k} \dot{\theta}\bigl((\theta^{-1})^{\circ (k-n+1)}(\phi(\lambda - n\tau^*))\bigr)} .
\end{equation}

As shown in~\cite{nah2020normalization, brunner2009time1}, for any admissible seed function, this recursion results in a time-transformation, $h$ which satisfies the conditions of Lemma 1. However, the seed function for a given time-varying delay is not uniquely defined. It is restricted only by the values of $\tau(0)$ and $\dot \tau(0)$ (and the choice of $\tau^*$). Furthermore, we note that the corresponding fixed delay representation in Eqn.~\eqref{eqn:02} includes the multiplicative time-varying parameter, $\dot h(\lambda)$. From Eqn.~\eqref{eq:deriv-prop1} we see that $\dot h(\lambda)$  will depend on the choice of seed function. Thus, while all seed functions will result in equivalent system representations, some seed functions may result in fixed-delay representations with large or unbounded parameter variation. In the following section, we consider a parameterization of seed functions which will allow us to search for seed functions which result in time transformations with certain desirable properties.

\section{Parameterization of Seed Functions}\label{sec5}

In this section, we provide an invertible affine map ($\mcl T$) between $L_2$ and the set of seed functions $S_{\tau_0',\tau_0,\tau^*}$. The key observation here is that any seed function, $\phi$, is uniquely determined by its third derivative, $\phi'''$. Following the main result in Thm.~\ref{param_oper}, we provide SoS conditions for $\phi$ to be increasing and show that the previously used quadratic seed function is a special case of the proposed parameterization.

Define operator $\mcl T$ as follows. 
\vspace{-2mm}
\begin{multline}
(\mcl T\nu)(\lambda) = -\tau_0 + \tau_0 \beta(\lambda) \label{eqn:op} \\
+ \int_{-\tau^*}^{0} K(\lambda, \eta) \nu(\eta)\, d\eta + \int_{-\tau^*}^{\lambda} \frac{(\lambda - \eta)^2}{2} \nu(\eta)\, d\eta,
\end{multline}
where
\begin{equation}
\beta(\lambda) = \frac{2(1 - \tau_0')}{(2 - \tau_0')\tau^*}(\lambda + \tau^*) 
+ \frac{\tau_0'}{(2 - \tau_0')(\tau^*)^2}(\lambda + \tau^*)^2, \label{eqn:beta}
\end{equation}
and
\begin{equation}
\begin{aligned}
K(\lambda, \eta) = \frac{1 - \tau_0'}{2 - \tau_0'} \left( -\eta - \frac{\eta^2}{\tau^*} \right)(\lambda + \tau^*) \label{eqn:K} \\
\qquad - \frac{1}{2 - \tau_0'} \left( \frac{\tau_0'\eta^2}{2(\tau^*)^2} - \frac{1 - \tau_0'}{\tau^*}\eta \right)(\lambda + \tau^*)^2.
\end{aligned}
\end{equation}
The following theorem shows that $\mcl T$ defines an invertible map from $L_2$ to $S_{\tau_0',\tau_0,\tau^*}$. 
\begin{theorem}\label{param_oper}
Given $\tau^*$, $\tau_0 >0$ and $\tau_0' < 1$. Let $S_{\tau_0',\tau_0,\tau^*}$ be defined as in Eqn.~\eqref{seed_set} and $\mcl T$ be defined as in Eqn.~\eqref{eqn:op}. Then: 
\begin{enumerate}
    \item For any $\phi \in S_{\tau_0',\tau_0,\tau^*}$, $\mcl T\phi'''= \phi$.
    \item For any $\nu \in L_2[-\tau^*, 0]$, $\mcl T\nu \in S_{\tau_0',\tau_0,\tau^*}$.
    \item For any $\nu \in L_2[-\tau^*, 0]$, $(\mcl T\nu)'''=\nu$.
\end{enumerate}
\end{theorem}
\begin{proof}
For 1), since $\phi \in S_{\tau_0',\tau_0,\tau^*} \subset \mcl H^3 [-\tau^*,0] \subset \mcl C^2[-\tau^*,0]$, we apply the Taylor formula with integral reminder~\cite{kountourogiannis2003derivation}: for any $f \in \mcl H^{k+1}([a,b])$ and  for any $x \in [a,b]$,
\begin{multline} 
\label{Eqn12}
f(x) = f(a) + f'(a)(x - a) + \frac{f''(a)}{2!}(x - a)^2 + \cdots \\
+ \frac{f^{(k)}(a)}{k!}(x - a)^k + \int_a^x f^{(k+1)}(\xi)\,\frac{(x - \xi)^k}{k!}\,d\xi.
\end{multline} 
Applying~\eqref{Eqn12} with  $f=\phi$, $a=-\tau^*$ and $k=2$ yeilds
\begin{equation}
\label{eq:taylor_phi_prime}
\phi'(0) = \phi'(-\tau^*) + \tau^* \phi''(-\tau^*) - \int_{-\tau^*}^{0} \eta \phi'''(\eta)\, d\eta,
\end{equation}
\begin{equation}
\label{eq:taylor_phi_zero}
\begin{aligned}
\phi(0) = \phi(-\tau^*) + \tau^* \phi'(-\tau^*) + \frac{1}{2}(\tau^*)^2 \phi''(-\tau^*) \\
+ \int_{-\tau^*}^{0} \frac{\eta^2}{2} \phi'''(\eta)\, d\eta, \quad \text{and}
\end{aligned}
\end{equation}
\begin{equation}
\label{eq:taylor_phi_lambda}
\begin{aligned}
\phi(\lambda) = \phi(-\tau^*) + (\lambda + \tau^*)\phi'(-\tau^*) + \frac{1}{2}(\lambda + \tau^*)^2 \phi''(-\tau^*)  \\
 + \int_{-\tau^*}^{\lambda} \frac{(\lambda - \eta)^2}{2} \phi'''(\eta)\, d\eta. 
 \end{aligned}
\end{equation}

$\phi \in S_{\tau_0',\tau_0,\tau^*}$ implies $\phi'(0) = \frac{\phi'(-\tau^*)}{(1-\dot{\tau}_0)}$, which combined with~\eqref{eq:taylor_phi_prime} yields
\begin{equation}
\label{eq:phi_dd_from_prime}
\phi''(-\tau^*) = \frac{1}{\tau^*} \left[ \frac{\tau_0'}{1 - \tau_0'} \phi'(-\tau^*)  + \int_{-\tau^*}^{0} \eta \phi'''(\eta)\, d\eta \right].
\end{equation}
Likewise combining $\phi(0) = 0$ and $\phi(-\tau^*) = -\tau_0$ with \eqref{eq:taylor_phi_zero} imply
\begin{equation}
\label{eq:phi_dd_from_zero}
\phi''(-\tau^*) = \frac{2}{(\tau^*)^2} \left[ \tau_0 - \tau^* \phi'(-\tau^*) 
 - \int_{-\tau^*}^{0} \frac{\eta^2}{2} \phi'''(\eta)\, d\eta \right].
\end{equation}
Substituting \eqref{eq:phi_dd_from_zero} into \eqref{eq:taylor_phi_prime} and rearranging yields
\begin{equation}
\label{eq:phi_p_at_minus}
\begin{aligned}
\phi'(-\tau^*) = \frac{1 - \tau_0'}{2 - \tau_0'} \left[ \frac{2\tau_0}{\tau^*} - \frac{1}{\tau^*} \int_{-\tau^*}^{0} \eta^2 \phi'''(\eta)\, d\eta \right. \\
\left. - \int_{-\tau^*}^{0} \eta \phi'''(\eta)\, d\eta \right].
\end{aligned}
\end{equation}
Substituting \eqref{eq:phi_p_at_minus} into \eqref{eq:phi_dd_from_prime}, we get:
\begin{multline}
\label{eq:phi_dd_final}
\phi''(-\tau^*) = \frac{1}{\tau^*} \left[ \frac{\tau_0'}{2 - \tau_0'} \left( \frac{2\tau_0}{\tau^*} - \frac{1}{\tau^*} \int_{-\tau^*}^{0} \eta^2 \phi'''(\eta)\, d\eta \right. \right. \\
\left. \left. - \int_{-\tau^*}^{0} \eta \phi'''(\eta)\, d\eta \right) + \int_{-\tau^*}^{0} \eta \phi'''(\eta)\, d\eta \right] \\
= \frac{2 \tau_0'}{2 - \tau_0'} \frac{\tau_0}{(\tau^*)^2} - \frac{\tau_0'}{2 - \tau_0'} \frac{1}{(\tau^*)^2} \int_{-\tau^*}^{0} \eta^2 \phi'''(\eta)\, d\eta \\
\qquad + \frac{2(1 - \tau_0')}{2 - \tau_0'} \frac{1}{\tau^*} \int_{-\tau^*}^{0} \eta \phi'''(\eta)\, d\eta.
\end{multline}
Now, substituting \eqref{eq:phi_p_at_minus} and \eqref{eq:phi_dd_final} into \eqref{eq:taylor_phi_lambda} and using the boundary condition $\phi(-\tau^*) =-\tau_0$, we obtain:
\begin{align*}
&\phi(\lambda) = -\tau_0 + (\lambda + \tau^*) \frac{1 - \tau_0'}{2 - \tau_0'} \left[ \frac{2\tau_0}{\tau^*} \right. \\
&\qquad  \qquad\qquad\left. - \frac{1}{\tau^*} \int_{-\tau^*}^{0} \eta^2 \phi'''(\eta)\, d\eta - \int_{-\tau^*}^{0} \eta \phi'''(\eta)\, d\eta \right] \\
&\qquad \quad + \frac{(\lambda + \tau^*)^2}{2} \left[ \frac{2 \tau_0'}{2 - \tau_0'} \frac{\tau_0}{(\tau^*)^2} \right. \\
&\qquad\qquad\qquad \left. - \frac{\tau_0'}{2 - \tau_0'} \frac{1}{(\tau^*)^2} \int_{-\tau^*}^{0} \eta^2 \phi'''(\eta)\, d\eta \right. \\
&\qquad \qquad\qquad\qquad \left. + \frac{2(1 - \tau_0')}{2 - \tau_0'} \frac{1}{\tau^*} \int_{-\tau^*}^{0} \eta \phi'''(\eta)\, d\eta \right] \\
&\qquad \quad + \int_{-\tau^*}^{\lambda} \frac{(\lambda - \eta)^2}{2} \phi'''(\eta)\, d\eta\\
&\qquad= -\tau_0 + \tau_0 \beta(\lambda) \\
&\qquad \quad  + \int_{-\tau^*}^{0} K(\lambda, \eta) \phi'''(\eta)\, d\eta + \int_{-\tau^*}^{\lambda} \frac{(\lambda - \eta)^2}{2} \phi'''(\eta)\, d\eta \\
&=(\mathcal{T}\phi''')(\lambda), \;\; \text{for all} \;\; \lambda\in[-\tau^*,0].
\end{align*}

Next, we establish 3) by showing that for any $\nu \in L_2[-\tau^*, 0]$, if $\phi = \mcl T\nu$, then $\phi''' = \nu$.
Recall from~\eqref{eqn:op} that
\vspace{-3mm}
\begin{multline*}
(\mcl T\nu)(\lambda) = -\tau_0 + \tau_0 \beta(\lambda) \\
+ \int_{-\tau^*}^{0} K(\lambda, \eta) \nu(\eta)\, d\eta + \int_{-\tau^*}^{\lambda} \frac{(\lambda - \eta)^2}{2} \nu(\eta)\, d\eta.
\end{multline*}
From the definitions of $\beta$ in~\eqref{eqn:beta} and $K$~in \eqref{eqn:K} , we observe that $\beta(\lambda)$ and $K(\lambda,\eta)$ are at most quadratic in $\lambda$. Hence,  $\beta'''(\lambda) = 0$ and $\partial_\lambda^3 K(\lambda,\eta) = 0$. Next, we note that
\begin{align*}
\frac{d}{d\lambda} \int_{-\tau^*}^\lambda \frac{(\lambda - \eta)^2}{2} \nu(\eta)\, d\eta &= \int_{-\tau^*}^\lambda (\lambda - \eta) \nu(\eta)\, d\eta, \\
\Rightarrow \quad \frac{d^2}{d\lambda^2} \int_{-\tau^*}^\lambda \frac{(\lambda - \eta)^2}{2} \nu(\eta)\, d\eta &= \int_{-\tau^*}^\lambda \nu(\eta)\, d\eta, \\
\Rightarrow \quad \frac{d^3}{d\lambda^3} \int_{-\tau^*}^\lambda \frac{(\lambda - \eta)^2}{2} \nu(\eta)\, d\eta &= \nu(\lambda).
\end{align*}
Thus, we conclude that $(\mcl T\nu)''' = \nu$.

For 2), suppose $\nu \in L_2[-\tau^*, 0]$. As per Eqn.~\eqref{seed_set}, we need to show that $\mcl T\phi \in H^3$, $(\mcl T\phi)(0)=0$, $(\mcl T\phi)(-\tau^*)=-\tau_0$, and $(\mcl T\phi)'(0)(1-\dot \tau)=(\mcl T\phi)'(-\tau^*)$. First, 3) implies $\phi = \mcl T \nu \in H^3$.  
Next, since $\beta(-\tau^*)=0$ and $K(-\tau^*,\eta)=0$, we have $(\mcl T\nu)(-\tau^*)=-\tau_0$. Third, since $\beta(0)=1$ and $K(0,\eta)=-\frac{\eta^2}{2}$, we have
\[
(\mcl T\nu)(0)=-\tau_0+\tau_0 + \int_{-\tau^*}^{0} \hspace{-.95mm} \frac{-\eta^2}{2} \nu(\eta)\, d \eta \\
+\int_{-\tau^*}^{0} \hspace{-.95mm} \frac{\eta^2}{2}\nu(\eta)\,d\eta = 0.
\]

For the final condition, differentiating~\eqref{eqn:op}, we obtain:
\begin{multline}
\label{deriv_op}
(\mcl T\nu)'(\lambda) = \tau_0 \beta'(\lambda) + \int_{-\tau^*}^{0} \partial_{\lambda} K(\lambda, \eta) \nu(\eta)\, d\eta \\
+ \int_{-\tau^*}^{\lambda} (\lambda-\eta) \nu(\eta)\, d \eta, \quad \forall \lambda \in [-\tau^*,0],
\end{multline}
\[
\text{where} \quad \beta'(\lambda)
= \frac{2\bigl(1-\tau_0'\bigr)}{\bigl(2-\tau_0'\bigr)\,\tau^*}
+ \frac{2\,\tau_0'}{\bigl(2-\tau_0'\bigr)\,(\tau^*)^{2}}\,(\lambda+\tau^*)
\]
\vspace{-2mm}
\begin{multline*}
\text{and} \;\; \partial_{\lambda} K(\lambda,\eta)
=
\frac{1-\tau_0'}{2-\tau_0'}
\left(-\eta-\frac{\eta^{2}}{\tau^*}\right)
\\
-
\frac{2}{2-\tau_0'}
\left(
\frac{\tau_0'\eta^{2}}{2(\tau^*)^{2}}
-
\frac{1-\tau_0'}{\tau^*}\eta
\right)
(\lambda+\tau^*).
\end{multline*}
We also note that
\[
\beta'(-\tau^*)=\frac{2(1-\tau_0')}{(2-\tau_0')\tau^*}
\]
and 
\[
\partial_\lambda K(-\tau^*,\eta)
	=
	\frac{1-\tau_0'}{2-\tau_0'}\Bigl(-\eta-\frac{\eta^2}{\tau^*}\Bigr)
\]
which implies
\begin{align}
&(\mcl T\nu)'(-\tau^*) = \tau_0 \beta'(-\tau^*) + \int_{-\tau^*}^{0} \partial_{\lambda} K(-\tau^*, \eta) \nu(\eta)\, d\eta \notag \\
&= \frac{1 - \tau_0'}{2 - \tau_0'} \left[ \frac{2\tau_0}{\tau^*} - \frac{1}{\tau^*} \int_{-\tau^*}^{0} \eta^2 \nu(\eta)\, d\eta  - \int_{-\tau^*}^{0} \eta \nu(\eta)\, d\eta \right].\label{eqn:Tv_at_tau_star2}
\end{align}
Similarly, since 
\[
\beta'(0)=\frac{2}{(2-\tau_0')\tau^*}
\]
and
\[
\partial_\lambda K(0,\eta)
=
\frac{1}{2-\tau_0'}\Bigl((1-\tau_0')\eta-\frac{\eta^2}{\tau^*}\Bigr),
\]
we have
\begin{align}
&(\mcl T\nu)'(0) = \tau_0 \beta'(0) + \int_{-\tau^*}^{0} \hspace{-1.2mm} \hspace{-1.2mm}\partial_{\lambda} K(0, \eta) \nu(\eta)\, d\eta - \int_{-\tau^*}^{0} \hspace{-1.2mm}\hspace{-1.2mm} \eta \nu(\eta)\, d \eta \notag \\
&= \frac{1}{2-\tau_0'} \left[ \frac{2\tau_0}{\tau^*}
- \frac{1}{\tau^*}\int_{-\tau^*}^{0}\eta^{2}\nu(\eta)\,d\eta - \int_{-\tau^*}^{0}\eta\nu(\eta)\,d\eta \right].
\label{eqn:Tv_at_tau_zero2}
\end{align}
Combining Eqn.~\eqref{eqn:Tv_at_tau_star2} and~\eqref{eqn:Tv_at_tau_zero2}, we obtain:
\[
\frac{(\mcl T\nu)'(-\tau^*)}{(\mcl T\nu)'(0)}
= 1-\tau_0'.
\]
We conclude that $\mcl T\nu \in S_{\tau_0',\tau_0,\tau^*}.$

\end{proof}

Having defined a parameterization of seed functions $\mcl T \nu$ for $\nu \in L_2$, we now observe that the simplest such seed function (for $\nu=0$) is precisely the quadratic seed function derived in~\cite[Thm. 1]{nah2020normalization}.
\begin{remark}
If $\nu(\eta) = 0$ for all $\eta \in [-\tau^*, 0]$, then the integral terms in $\mcl T \nu$ (See Eqn.~\eqref{eqn:op}) vanish, and 
for $\lambda \in [-\tau^*,0]$,
\vspace{-2mm}
\begin{equation}
\label{trivial}
\phi(\lambda) = -\tau(0) + \tau(0) \beta(\lambda),
\end{equation}
where  $\beta(\lambda)$ in~\eqref{eqn:beta} is a quadratic polynomial in $\lambda$.
This  quadratic seed function coincides with the choice  used in ~\cite[Thm. 1]{nah2020normalization}.
\end{remark}

Although Thm.~\ref{param_oper} ensures that for any $\nu \in L_2$, $\phi = \mcl T\nu \in S_{\tau_0',\tau_0,\tau^*}$, admissibility of a seed function further requires $\phi'$ to be non-negative. Such a condition can be imposed upon the seed parameter, $\nu$ as follows.
\begin{corollary}\label{cor:SOS}
Given $\tau_0', \tau_0,\tau^*$, let  $\mcl T$  be as in Eqn.~\eqref{eqn:op}. For any $\nu\in L_2[-\tau^*,0]$, $(\mcl T\nu)'(\lambda)>0$ for all $\lambda \in [-\tau^*,0]$ if and only if 
\vspace{-4mm}
\begin{multline*}
\label{pos_cond}
\tau_0 \beta'(\lambda) + \int_{-\tau^*}^{0} \partial_{\lambda} K(\lambda, \eta) \nu(\eta)\, d\eta  \\
+ \int_{-\tau^*}^{\lambda} (\lambda-\eta) \nu(\eta)\, d \eta >0,
\end{multline*}
for all $\lambda \in [-\tau^*,0]$.
\end{corollary}

\begin{proof}
Follows from the definition of $\mcl T$ in Eqn.~\eqref{deriv_op}. 
\end{proof}
If the seed parameters are polynomial, then SoS optimization can be used to enforce the non-negativity condition in Cor.~\ref{cor:SOS}. Note that since the non-negativity condition here is univariate, the corresponding univariate  SoS non-negativity test is both \emph{necessary} and \emph{sufficient}. 
\section{Numerical Comparison of Seed Functions}\label{sec6}
\begin{figure}[t]
\vspace{2mm}
    \centering
    \begin{subfigure}[t]{0.49\linewidth}
        \centering
        \includegraphics[width=\linewidth]{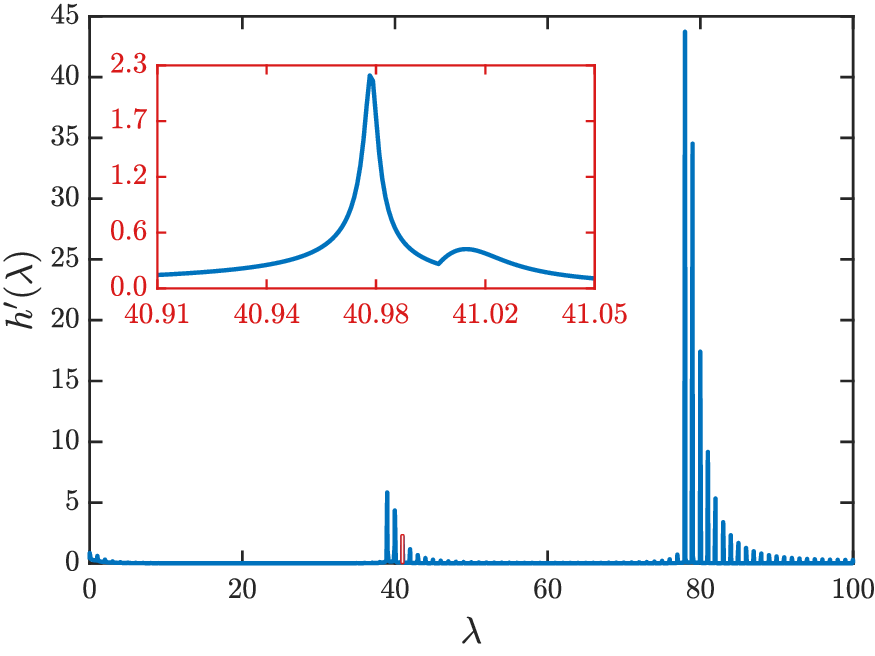}
        \subcaption{}
        \label{fig:two-up:a}
    \end{subfigure}\hfill
    \begin{subfigure}[t]{0.49\linewidth}
        \centering
        \includegraphics[width=\linewidth]{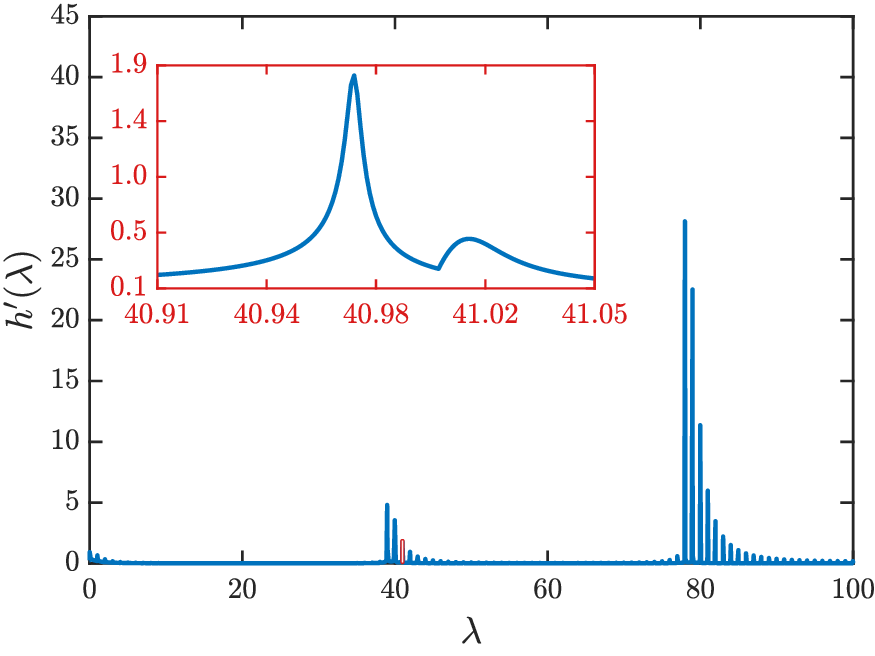}
        \subcaption{}
        \label{fig:two-up:b}
    \end{subfigure}\hfill
    \begin{subfigure}[t]{0.49\linewidth}
        \centering
        \includegraphics[width=\linewidth]{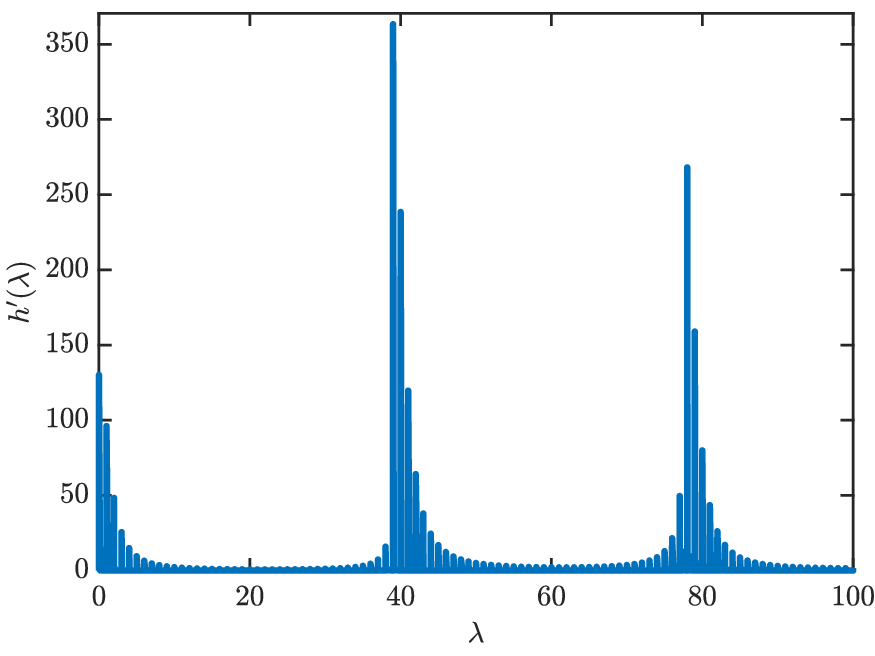}
        \subcaption{}
        \label{fig:two-up:c}
    \end{subfigure}
     \caption{The derivative $h'(\lambda)$ of the time transformation for (a) quadratic seed, (b) affine plus sinusoidal seed, and (c) exponential seed, with different maximum values of $h'$ for horizon of 100-- illustrating the impact of seed function choice on stability margins for a delay $\tau(t) = (\frac{1}{2\pi} - 0.001) \sin(2\pi t) + (\frac{1}{2\pi} + 0.001)$.}
  \label{fig:two-up}
\end{figure}

\begin{figure}
\vspace{2mm}
  \centering
    \begin{subfigure}[t]{0.49\linewidth}
        \centering
        \includegraphics[width=\linewidth]{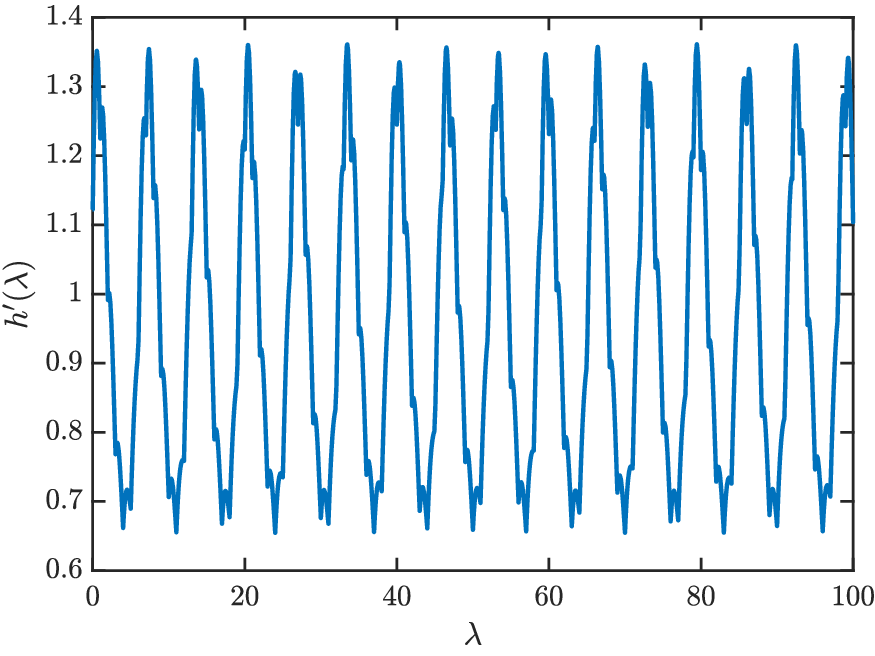}
        \subcaption{}
        \label{fig:two-up:d}
    \end{subfigure}\hfill
    \begin{subfigure}[t]{0.49\linewidth}
        \centering
        \includegraphics[width=\linewidth]{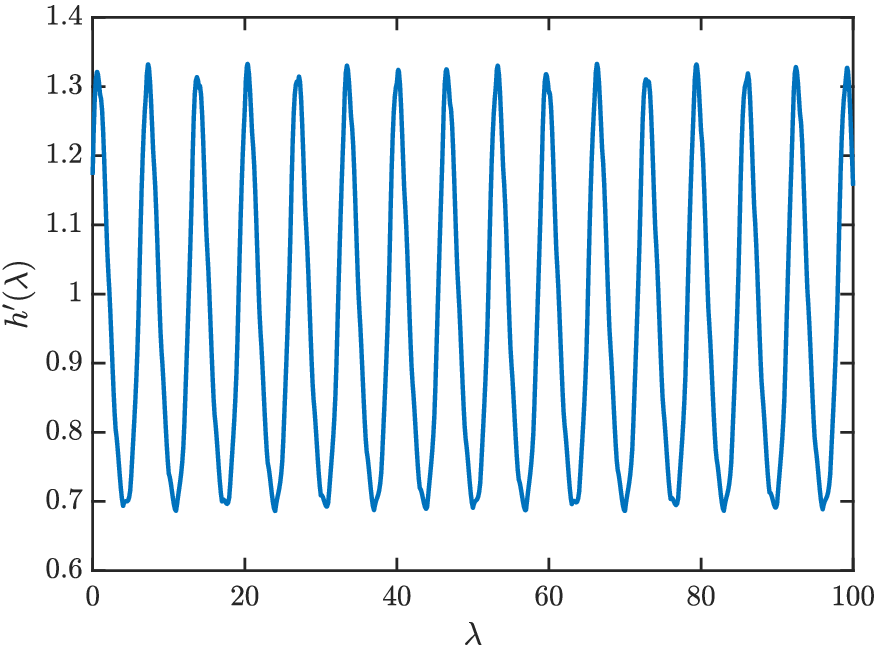}
        \subcaption{}
        \label{fig:two-up:e}
    \end{subfigure}\hfill
    \begin{subfigure}[t]{0.49\linewidth}
        \centering
        \includegraphics[width=\linewidth]{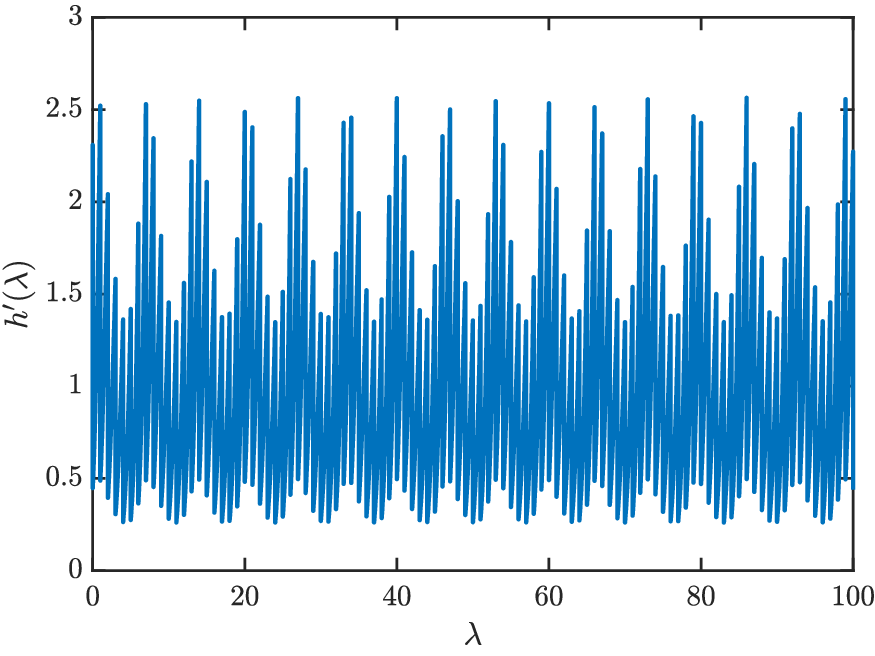}
        \subcaption{}
        \label{fig:two-up:f}
    \end{subfigure}
    \caption{The derivative $h'(\lambda)$ of the time transformation for (a) quadratic seed, (b) affine plus sinusoidal seed, and (c) exponential seed, with different maximum values of $h'$ for horizon of 100 for the delay $\tau(t)=1+0.3sin(t)$.}
    \label{fig:example2}
\end{figure}

For a given time-varying delay $\tau(t)$, the choice of seed function  significantly influences the time transformation $h$ and its derivative $h'$, potentially complicating analysis of the fixed-delay, parameter-varying representation. In this section, we compare the effect of three choices of seed parameter on the derivative of the time-transformation, $h'$.

Specifically, we consider two examples of periodic time-varying delay, and for each example, examine three candidate seed parameters. For each resulting seed function, we construct $h'(\lambda)$ over an extended interval. 

Specifically, recall that for $\theta(t)=t-\tau(t)$,
\vspace{-2mm}
\begin{equation*}
h(\lambda) := (\theta^{-1})^{\circ k} \bigl( \phi(\lambda - k\tau^*) \bigr) \qquad \lambda \in [(k-1)\tau^*, k\tau^*].
\end{equation*}
Because there is no analytic expression for $\theta^{-1}$, the $h(\lambda)$ and $\dot h(\lambda)$ are computed pointwise in $\lambda$. That is, for every given $ \lambda \in [(k-1)\tau^*, k\tau^*]$ with $k\in \mathbb{N}$, we initialize 
 $h_0:=\phi(\lambda-k\tau^*)$ and for $j=0, \cdots, k$,  compute $h_{j+1}:=(\theta^{-1})(h_{j})$ where $\theta^{-1}$ is evaluated numerically using Newton iteration -- yielding $h_j=(\theta^{-1})^{\circ j} \bigl( \phi(\lambda - k\tau^*) \bigr)$ and $h_k=h(\lambda)$. The $h_j$ are then used to compute $h'(\lambda)$ as
\[
h'(\lambda) = \frac{\phi'(\lambda-k\tau^*)}{\prod^k_{j=1} \theta'(h_j)}
\]
For all examples, a time horizon of $\lambda \in [0,100]$ is used.

\paragraph{Example 1}
First consider time-varying delay $\tau(t) = \gamma_0 \sin(2\pi t) + \gamma_1$, where $\gamma_0 = \frac{1}{2\pi} - 0.001$ and $\gamma_1 = \frac{1}{2\pi} + 0.001$. Here $\tau(0)=\gamma_1$, $ \tau'(0)=2\pi \gamma_0$ and we choose $\tau^*=1$, corresponding to the period of delay.

We now select three seed parameters $\nu\in L_2$: 
\vspace{-2mm}
\begin{align*}
\nu_1(\lambda)&=0,\qquad \nu_2(\lambda) = \frac{8 \gamma_1 }{1 - e^{-2}} e^{2 \lambda}\\
\nu_3(\lambda) &= -\Lambda_1 \left(\frac{\pi}{2 }\right)^3 \sin\left(\pi (\lambda + 1)\right)
\end{align*}
where
\vspace{-2mm}
\[
\Lambda_1 = \frac{4\pi \gamma_0 \gamma_1}{\pi-2\pi^2\gamma_0+2\gamma_1} .
\]
The corresponding quadratic, exponential and affine plus sinusoidal seed functions are then 
\begin{equation*}
\label{quadratic_seed}
(\mcl T\nu_1)(\lambda) = -\tau(0) + \tau(0) \beta(\lambda), \;\;
(\mcl T\nu_2)(\lambda) = \frac{\tau(0) (e^{2 \lambda} - 1)}{1 - e^{-2}}
\end{equation*}
and
\begin{equation*}
\label{affine_sinusoidal}
(\mcl T\nu)(\lambda) =  \Lambda_2 - \tau(0) + \Lambda_2 \lambda + \Lambda_1 \left(1 - \cos\left(\frac{\pi}{2} (\lambda + 1) \right)\right),
\end{equation*}
where
\[
 \Lambda_2 = \frac{0.002 \pi^2 \gamma_1}{0.002 \pi^2+ 2\gamma_1}
\]
 and where $\beta(\lambda)$ is defined in Eqn.~\eqref{eqn:beta} for $\tau_0=\gamma_1$ and $ \tau_0'=2 \pi \gamma_0$. It is easily verified that the seeds $ \mcl T(\nu_i) \in S_{\tau_0',\tau_0,\tau^*}$ are monotonic.

\paragraph{Example 2} In the second example, we use a time-varying delay $\tau(t)=1+0.3\sin(t)$. Here $\tau(0)=1$, $\tau'(0)=0.3$ and we choose $\tau^*=1$, which is a fraction $\frac{1}{2\pi}$ the period of delay. In this case, we again use slightly modifed version of $\nu_1,\nu_2$ and $\nu_3$ from the first example:
\vspace{-3mm}
\begin{align*}
\nu_1(\lambda)&=0,\qquad \nu_2(\lambda) = \frac{8}{1 - e^{-2}} e^{2 \lambda}\\
\nu_3(\lambda) &= -\Lambda_1 \left(\frac{\pi}{2 }\right)^3 \sin\left(\pi (\lambda + 1)\right)
\end{align*}
where
\vspace{-2mm}
\[
\Lambda_1 = \frac{2.6}{0.7 \pi + 2}, \quad \Lambda_2 = \frac{0.6}{0.7\pi +2}.
\]
\paragraph{Analysis} The derivative of the time-transformation, $ h'(\lambda)$ corresponding to the time-delays and seed parameters from Example 1 can be found in Figure~\ref{fig:two-up}. For Example 2, $ h'(\lambda)$ is shown in Figure~\ref{fig:example2}. Among the tested seed functions, the affine-plus-sinusoidal parameter produces the smallest bounds on $h'$, followed by the quadratic function, while the exponential function results in the largest bounds. The suitability of the affine-plus-sinusoidal may be due to structural similarity between seed function and time-varying delay, suggesting the possibility of seed functions whose corresponding time-transformation admit global bounds on the derivative. These results also highlight the importance of careful selection of seed function to minimize parameter variation, thereby enhancing stability analysis in small-gain or IQC frameworks.

\section{Conclusion}
Systems with time-varying delay can be equivalently represented by systems with fixed delay and multiplicative parameter variation. However, this representation is not unique, being parameterized by a time-transformation which is, in turn, defined by a seed function. The choice of seed function has a significant effect on the parameter variation of the resulting constant-delay representation. In this paper, we show that the set of admissible seed functions can be parameterized by the set of seed parameters consisting of the space of $L_2$ functions. We then examine several choices of seed parameter and show how this choice impacts the parameter variation of the resulting fixed-delay representation. This parameterization may allow for direct optimization of seed functions to minimize parameter variation for use in small gain or IQC type stability conditions.
\bibliographystyle{IEEEtran}
\bibliography{References}

\end{document}